\documentclass[11pt,reqno]{amsart}
\usepackage[margin=1in]{geometry}
\usepackage{amssymb,amsmath,amsthm}
\usepackage{graphicx}
\usepackage{subcaption}
\usepackage{fullpage}
\usepackage{scrextend}
\usepackage{siunitx}
\usepackage{mathtools}
\usepackage[dvipsnames]{xcolor}
\usepackage{microtype}
\usepackage{enumerate}
\usepackage{tikz,calc}
\usepackage{tikz-cd}
\usetikzlibrary{automata,positioning}
\usepackage{comment}
\usetikzlibrary{calc}
\usepackage{epsfig,graphicx}
\usepackage{listings}
\usepackage{enumitem}

\usepackage{float}
\usepackage{bbm}
\usepackage{url}
\usepackage[colorlinks=true, pdfstartview=FitH, linkcolor=NavyBlue, citecolor=RawSienna, urlcolor=RoyalBlue]{hyperref} 
\usepackage[capitalise,compress,nameinlink,noabbrev]{cleveref} 

\usepackage[backend=biber, style=numeric, firstinits=true, url=false,eprint=true,maxbibnames=99,doi=true]{biblatex}
\patchcmd{\subsubsection}{\itshape}{\bfseries}{}{}

\newcommand{\defeq}{\vcentcolon=}

\renewcommand{\leq}{\leqslant}

\renewcommand{\geq}{\geqslant}

\newcommand{\R}{\mathbb{R}}

\newcommand{\Mod}[1]{\ (\text{mod}\ #1)}
\renewcommand{\Mod}[1]{{\ifmmode\text{\rm\ (mod~$#1$)}\else\discretionary{}{}{\hbox{ }}\rm(mod~$#1$)\fi}}

\newcommand{\allnotes}[1]{}
\renewcommand{\allnotes}[1]{\textit{#1}}

\newtheorem{theorem}{Theorem}[section]
\newtheorem{prop}[theorem]{Proposition}
\newtheorem{lemma}[theorem]{Lemma}

\newtheorem{conjecture}[theorem]{Conjecture}

\newtheorem{question}[theorem]{Question}

\theoremstyle{definition} 
\newtheorem{defn}[theorem]{Definition}

\numberwithin{theorem}{section}

\allowdisplaybreaks

\newcommand{\s}{\sigma}
\newcommand{\ta}{\tau}

\DeclareMathOperator{\E}{\mathbb{E}}

\DeclarePairedDelimiter\ceil{\lceil}{\rceil}
\DeclarePairedDelimiter\floor{\lfloor}{\rfloor}

\begin{document}

\parskip \smallskipamount

\title{On the largest size of sum-free sets in symmetric regions}

\author{Anubhab Ghosal}
\address{Mathematical Institute \\ University of Oxford \\ Oxford OX2 6GG \\ UK }
\email{ghosal@maths.ox.ac.uk}
\author{Dmitry Tsarev}
\address{Mathematical Institute \\ University of Oxford \\ Oxford OX2 6GG \\ UK }
\email{tsarev@maths.ox.ac.uk}


\begin{abstract}
    A subset $S$ of a group $G$ is said to be sum-free (resp. $\Delta$-free) if there are no solutions to $a+b=c$ (resp. $a+b+c=0$) with $a,b,c\in S$. For a convex region $R\subset\mathbb{R}^d$, let $\sigma(R)$ denote the maximal proportion of the volume of $R$ that a sum-free subset of $R$ can occupy. 
    
    We prove that $\sigma([-1,1]^d)=1/2$. Our proof employs a careful application of the Brunn-Minkowski inequality. Moreover, for the $d$-dimensional Euclidean ball $\mathbb{B}^d(0,1)$, we show that $\sigma(\mathbb{B}^d(0,1))\leq 1/2+o_d(1)$. We present two arguments for this. The first combines some routine harmonic analysis on the sphere with known bounds on values of the ultraspherical polynomials. The second more elementary argument proceeds by establishing that the maximal $\Delta$-free subset of the unit sphere $\mathbb{S}^{d-1}$ occupies $1/2+O(d^{-1})$ of the sphere's surface measure. This answers a question raised by Bukh.
\end{abstract}

\maketitle

\section{Introduction}
A subset $S$ of an abelian group $G$ is said to be \emph{sum-free} if there are no solutions to $a+b=c$ with $a,b,c\in S$. Following the highly influential paper of Erd\H{o}s~\cite{erdos1965extremal} more than half a century ago, the study of extremal properties of sum-free sets has received much attention. For example, Diananda and Yap~\cite{diananda1969maximal}, and Green and Ruzsa~\cite{green2005sum} have determined the size of the largest sum-free subsets of all finite abelian groups; see ~\cite{taovusurvey} for a survey on sum-free sets in groups.

A more recent line of inquiry has sought to determine the largest volume of sum-free subsets of regions in $\mathbb{R}^d$. For a convex region $R\subset\mathbb{R}^d$, let $\sigma(R)$ denote the maximal proportion of the volume of $R$ that a sum-free subset of $R$ can occupy. The (discrete version of the) problem of determining $\sigma([0,1]^2)$ was communicated by Aydinian to Serra and posed as an open problem at the 19th British Combinatorial Conference in 2001~\cite{cameron2005research}. Note that the sum-free subset
\[
\{(x,y)\in [0,1]^2: 0.8\leq x+y<1.6\}
\]
shows that $\sigma([0,1]^2)\geq 0.6$. Cameron~\cite{cameron2005research, cameronnote} conjectured that, indeed $\sigma([0,1]^2)=0.6$, and this was established in 2017 by Elsholtz and Rackham~\cite{Elsholtz2017}.

In higher dimensions, the problem of determining $\sigma([0,1]^d)$ was also raised by Aydinian and Cameron~\cite{cameronnote}, and highlighted by Green~\cite{Green}. To be completely accurate, the discrete version of this problem was posed, which is strictly harder. However, the main difficulty in solving these problems is captured well by the ``nicer" continuous versions. Similarly as above, we have the lower bound
\[
\sigma([0,1]^d)\geq c_d^*\defeq\max\{\text{Vol}(\{x\in [0,1]^d:u\leq\sum_ix_i<2u\}):u\in [0,d]\}.
\]
In a series of recent breakthroughs, Lepsveridze and Sun~\cite{Lepsveridze2026} and Keevash and Lim~\cite{keevash2026largest}, have established that indeed $\sigma([0,1]^d)=c_d^*$ holds for $d\in\{3,4\}$ and $d\geq 5$ respectively.

In this paper, we study the extremal density $\sigma(R)$ for symmetric convex regions $R\subset\mathbb{R}^d$. When $d=1$, note that
\[
[-1,-0.5]\cup [0.5,1]
\]
is a sum-free subset of $[-1,1]$, so $\sigma([-1,1])\geq 0.5$. On the other hand,
\[
\sigma([-1,1])\leq \sigma([0,1])=0.5.
\]
Therefore $\sigma([-1,1])=0.5$. By lifting the construction to higher dimensions, we get that $\sigma([-1,1]^d)\geq 0.5$ as $([-1,-0.5]\cup [0.5,1])\times [-1,1]^{d-1}$ is sum-free. Our first theorem shows that this is optimal.

\begin{theorem}
\label{thm: sum-free cube}
    For all $d\in\mathbb{N}$, we have that
    $\s([-1,1]^d)=\frac{1}{2}$.
\end{theorem}

Our proof is a short but careful application of the Brunn-Minkowski inequality. Indeed, a naive application of Brunn-Minkowski would show that $\operatorname{Vol}(A+A)\geq 2^d\operatorname{Vol}(A)$ for $A\subset[-1,1]^d$. As $A+A\subset [-2,2]^d$, for sum-free $A\subset[-1,1]^d$, this allows us to conclude that $\sigma([-1,1]^d)\leq \frac{2^d}{2^d+1}$, which is far from optimal.

What about other convex symmetric regions $R$? We believe that the above upper bound continues to hold.
\begin{conjecture}
\label{main conj}
    For any symmetric, convex region $R\subset \mathbb{R}^d$, we have that $\sigma(R)\leq \frac{1}{2}$.
\end{conjecture}
If it holds, the conjecture would show that the cube $[-1,1]^d$ is extremal among convex symmetric regions in the sense of having the largest largest sum-free subset. Perhaps surprisingly, we are unable to prove Conjecture \ref{main conj} even for the (centred) euclidean ball $\mathbb{B}^d(0,1)$. Our second result proves the following asymptotic upper bound.

\begin{prop}
\label{fourier prop}
    There exist absolute constants $c>0$ and $d_0\in \mathbb{N}$ so that 
    \[
    \sigma(\mathbb{B}^d(0,1))\leq \frac{1}{2}+\exp(-cd)
    \]
    holds for all $d\geq d_0$.
\end{prop}

We prove this by showing that a sum-free subset of the sphere occupies at most $1/2+\exp(-cd)$ of its surface measure for $d\geq d_0$. We do so by combining some routine harmonic analysis on the sphere with some bounds on Gegenbauer polynomials due to Castro-Silva, Filho, Slot and Vallentin~\cite{Castro-Silva2022-mn}. By considering  Erdos' \emph{middle-thirds} construction~\cite{erdos1965extremal}, one can see that $\sigma(\mathbb{B}^d(0,1))\geq \frac{1}{3}$. It is possible that this lower bound is closer to the truth and consequently, Proposition \ref{fourier prop} is far from optimal.

Bukh~\cite{bukhproblem} raised the following related question. A subset $S\subset G$ is said to be \emph{$\Delta$-free} if there are no solutions to $a+b+c=0$. Let $\tau(\mathbb{S}^{d-1})$ denote the maximal proportion of the surface measure of $\mathbb{S}^{d-1}$ that a $\Delta$-free subset of $\mathbb{S}^{d-1}$ can occupy. Note that, by considering a hemisphere, one obtains $\tau(\mathbb{S}^{d-1})\geq 0.5$. Bukh~\cite{bukhproblem} asked if $\tau(\mathbb{S}^{d-1})= 0.5$, or even if $\tau(\mathbb{S}^{d-1})= 0.5+o_{d\to \infty}(1)$. We show that the latter holds.

\begin{theorem}
\label{thm: trangle-free sphere}
    For any $d\in \mathbb{N}$, we have that 
    \[
        \tau(\mathbb{S}^{d-1})\leq 
        \begin{cases}
            \frac{1}{2}+\frac{1}{2d},&\text{ if $d$ is odd},\\
            \frac{1}{2}+\frac{1}{2(d+1)},&\text{ if $d$ is even}.\\
        \end{cases}
    \]
\end{theorem}

Bukh~\cite{bukhproblem} mentions that the extremal density of a $\Delta$-free subset of $\mathbb{R}^d$ equipped with the Gaussian measure is known to equal a half. Noam Lifshitz (personal communication) informed us that the known proof is a straightforward deduction from a Brunn-Minkowski type inequality in Gaussian spaces due to Borell~\cite{Borell2008-zr}. In this paper, we present a different argument (see proof of Proposition \ref{prop: gaussian triangle-free}) which motivates our proof of Theorem \ref{thm: trangle-free sphere}.

Our proof of Theorem \ref{thm: trangle-free sphere} brings to light a connection between sum-free (resp. $\Delta$-free) sets and triangle-free subgraphs of unit distance (di)graphs in $\mathbb{R}^d$. We say that a (di)graph $G$ on vertex set $V\subset\R^d$ is unit distance if every edge $xy\in E(G)$ satisfies $\|x-y\|_2=1$. Let $\bar{t}(G)$ denote the maximum number of edges in a subgraph of $G$ without any directed (cyclic) triangles. For $d\in \mathbb{N}$, define 
\[
G(d)\defeq \operatorname{inf}\left\{\frac{\bar{t}(G)}{|E(G)|}\;:\;\text{$G$ is a finite unit distance digraph in $\R^d$}, |E(G)|>0\right\}.
\]
We will prove the following proposition.

\begin{prop}
\label{prop: bound of tau by G}
    For any $d\in \mathbb{N}$, we have that $\tau(\mathbb{S}^{d-1})\leq G(d)$.
\end{prop}

We believe the following bound holds for $G(d)$.

\begin{conjecture}
\label{conj: bound on $G$}
    There exist absolute constants $c>0$ and $d_0\in \mathbb{N}$ so that $G(d)\leq \frac{1}{2}+\exp(-cd)$ holds for all $d\geq d_0$.
\end{conjecture}

If the conjecture holds, Proposition \ref{prop: bound of tau by G} would imply much stronger bounds on $\tau(\mathbb{S}^{d-1})$ than Theorem \ref{thm: trangle-free sphere}. We discuss this further in Section \ref{subsection: conc1}. We will also prove the following lower bound on $G(d)$, which shows that one cannot hope for better than an exponential error using this approach.

\begin{prop}
\label{prop: G lower bound}
    There exists an absolute constant $C>0$ such that $G(d)\geq \frac{1}{2}(1+\exp(-Cd))$ holds for all $d\in\mathbb{N}$.
\end{prop}

The rest of the paper is organised as follows. In Section \ref{sec: proofs}, we present the proofs of all our results. In Section \ref{sec: conc}, we present some more open problems and make some concluding remarks.

\subsection*{Note.}
Theorem~\ref{thm: trangle-free sphere} was also proved independently by D\'ucz~\cite{ducz2026} and Ge and Xu~\cite{xu2026}. Our proof was obtained in September 2025, and was presented at a public seminar at IMPA, Rio de Janeiro, on March 27, 2026.

\section*{Acknowledgements}

The authors thank Ben Green, Noam Lifshitz, Rob Morris and Carl Schildkraut for useful conversations. AG is supported by the Clarendon Fund and Oxford Ryniker Lloyd Graduate Scholarship. DT is supported by EPSRC grant EP/W523781/1.

\section{Proofs}
\label{sec: proofs}

We start by setting up some notation.

\begin{defn}
    For a probability measure $\mu$ on $\mathbb{R}^d$, let
    \[
        \s(\mu)\defeq\sup\{\mu(S)\; |\;S\subset \R^d\text{ is Borel and sum-free}\},\ 
        \ta(\mu)\defeq\sup\{\mu(S)\;|\;S\subset \R^d\text{ is Borel and $\Delta$-free}\}.
    \]
    For a bounded region $R\subset \R^d$, we let $\s(R)$ and $\ta(R)$ denote the quantities $\s(\mu_R)$ and $\ta(\mu_R)$, where $\mu_R$ is the uniform probability measure on $R$. In particular, when $R$ is a convex body in $\R^d$, $\mu_R$ is the normalised restriction of the Lebesgue measure $\lambda^d$ to $R$; when $R=\mathbb{S}^{d-1}$ is the unit sphere in $\R^d$, $\mu_R$ is the normalised surface measure, and when $R$ is finite $\mu_R$ is the normalised counting measure.
\end{defn}

\subsection{Sum-free subsets of symmetric cubes}

To prove Theorem \ref{thm: sum-free cube}, we need to recall the Brunn-Minkowski inequality~\cite{gardner2002brunn}. It 
states that
\begin{equation}
\label{eq: B-M}
    \lambda(A+B)^{1/d}\geq \lambda(A)^{1/d}+\lambda(B)^{1/d},
\end{equation}
whenever $A$, $B$ and $A+B$ are measurable subsets of $\mathbb{R}^d$. While the sum of Lebesgue measurable sets isn't necessarily measurable, the sum of Borel sets is always measurable (see ~\cite[Section 4]{gardner2002brunn}). We are now ready to prove the theorem.

\begin{proof}[Proof of Theorem \ref{thm: sum-free cube}]
    Recall that $\sigma([-1,1]^d)\geq 1/2$, so it suffices to prove the upper bound.
    Let $S$ be a sum-free Borel subset of $[-1,1]^d$. By the pigeonhole principle, there is an orthant $Q$ such that $S'\defeq S\cap Q$ has measure
    \begin{align}
    \label{eq: lowerbd sharp eq1}
        \lambda(S')\geq 2^{-d}\lambda(S).
    \end{align}
    By the Brunn-Minkowski inequality with $A=S'$ and $B=-S'$, we have that
    \begin{align}
    \label{eq: lowerbd sharp eq2}
        \lambda(S'-S')\geq 2^d\lambda(S').
    \end{align}
    Now, note that the cube is self-similar and so $(Q\cap [-1,1]^d)-(Q\cap [-1,1]^d)=[-1,1]^d$ for any orthant $Q$. Therefore, $S'-S'\subset [-1,1]^d$ and as $S$ is sum-free, $S\cap (S'-S')=\varnothing$. It follows that
    \begin{align}
    \label{eq: lowerbd sharp eq3}
        \lambda(S)+\lambda(S'-S')\leq \lambda([-1,1]^d).
    \end{align}
    Combining Inequalities \eqref{eq: lowerbd sharp eq1}, \eqref{eq: lowerbd sharp eq2}, and \eqref{eq: lowerbd sharp eq3}, we have that $2\lambda(S)\leq\lambda([-1,1]^d)$ and so $\s([-1,1]^d)\leq \frac{1}{2}$.
\end{proof}

\subsection{$\Delta$-free subsets under Gaussian Measure}

Next, we will prove the extremal result for $\Delta$-free sets under the Gaussian measure. We need the following directed version of Mantel's theorem due to Brown and Harary~\cite{BrownHarary1970ExtremalDigraphs}, which we prove for completeness. The proof closely follows the usual inductive proof of Mantel's theorem where one ``grabs the graph by an edge".

\begin{lemma}\label{lem:digraph-no-triangle}
    If $G$ is a digraph on $n$ vertices without a directed triangle, then $|E(G)|\leq 2\floor*{\frac{n^2}{4}}$.
\end{lemma}
\begin{proof}
    We induct on $n$. The claim holds for $n=0$ and $n=1$. Suppose $n>1$. We can assume that $G$ has edges $xy$ and $yx$ for some $x,y\in V(G)$. Indeed, if not, then $|E(G)|\leq \binom{n}{2}\leq 2\floor*{\frac{n^2}{4}}$, and we would be done. Let $G'=G[V(G)\setminus\{x,y\}]$.
    
    Note that for any $z\in V(G)\setminus\{x,y\}$, as the directed triangle $xyz$ is absent in $G$, at most one of the edges $yz$ and $zx$ occurs in $G$. It follows that $d^{+}_G(y)+d^-_G(x)\leq n$. Similarly, $d^{-}_G(y)+d^+_G(x)\leq n$. Finally, by the induction hypothesis, we have that $|E(G')|\leq 2\floor*{\frac{(n-2)^2}{4}}$. It follows that
    \[
    |E(G)|=d^+_G(y)+d^-_G(y)+d^+_G(x)+d^-_G(x)-2+|E(G')|\leq 2n-2+2\floor*{\frac{(n-2)^2}{4}}=2\floor*{\frac{n^2}{4}}.\]
    
\end{proof}

We use $\gamma_{\sigma^2}^d$ to denote the Gaussian probability measure on $\R^d$ with mean $0$ and variance $\sigma^2$, that is, the probability measure with density given by $\frac{\mathrm{d}\gamma_{\sigma^2}^d}{\mathrm{d}\lambda}(x)=(2\pi\sigma^2)^{-d/2}\exp\left(-\frac{1}{2\sigma^2}\|x\|^2\right)$. We are now ready to determine the extremal Gaussian measure of a $\Delta$-free subset of $\mathbb{R}^d$. As mentioned in the introduction, this result is already known, but our proof is novel.
    
\begin{prop}
\label{prop: gaussian triangle-free}
    Let $\gamma=\gamma_1^d$ be the standard Gaussian measure on $\R^d$. Then $\ta(\gamma)=\frac{1}{2}$.
\end{prop}
\begin{proof}
    Note that $\{x\in \mathbb{R}^d:x_1>0\}$ is $\Delta$-free, so it suffices to prove the upper bound.
    Let $S$ be a $\Delta$-free set in $\R^d$ and let $\chi$ be its indicator function. Fix a positive integer $n$, and let $\{X_i\}_{i\in[n]}$ be a collection of independent random vectors in $\R^d$, each with law given by $\gamma_{\frac{1}{2}}^d$. Observe that, for all $i\neq j$, $X_i-X_j$ has law $\gamma$ and so
    \[
        \E\sum_{i\neq j}\chi(X_i-X_j)=n(n-1)\cdot\gamma(S).
    \]
    Consider the digraph $G$ on vertex set $[n]$, with directed edge $ij$ present in $G$ if and only if $X_i-X_j\in S$. As $S$ is $\Delta$-free, $G$ does not contain a directed triangle. It follows by Lemma~\ref{lem:digraph-no-triangle} that  
    \[
        \sum_{i\neq j}\chi(X_i-X_j)=|E(G)|\leq 2\floor*{\frac{n^2}{4}}\leq \frac{n^2}{2}.
    \]
    It follows that $\gamma(S)\leq \frac{1}{2}+\frac{1}{2(n-1)}$. Since $n$ was chosen arbitrarily, we have $\gamma(S)\leq\frac{1}{2}$.
\end{proof}

\subsection{$\Delta$-free subsets of spheres}
\begin{proof}[Proof of Theorem \ref{thm: trangle-free sphere}]
    Let $S=\{s_1,\ldots,s_{d+1}\}\subset\mathbb{R}^d$ denote a set of $d+1$ distinct points, such that all pairwise distances equal unity. In other words, $S$ is the set of vertices of some unit regular $d$-dimensional simplex. Further, let
    \[
    D=(S-S)\setminus\{0\}
    \]
    Note that $D\subset \mathbb{S}^{d-1}$. By averaging over a uniform rotation $\rho\in \operatorname{SO}(d)$, it suffices to show that $\tau(D)\leq \frac{1}{2}+\frac{1}{2d}$ for odd $d$ and $\tau(D)\leq \frac{1}{2}+\frac{1}{2(d+1)}$ for even $d$. Let $A\subset D$ be a $\Delta$-free subset of $D$. Let $G$ be a digraph on vertex set $[d+1]$ such that $uv\in E(G)$ if and only if $s_u-s_v\in A$. Observe that, since $A$ is $\Delta$-free, $G$ does not contain a directed triangle. By Lemma~\ref{lem:digraph-no-triangle}, it follows that $|E(G)|\leq2\floor{\frac{(d+1)^2}{4}}$, and so we have that
    \[
    \tau(D) = \frac{|A|}{d(d+1)} \leq \frac{2}{d(d+1)}\left\lfloor\frac{(d+1)^2}{4}\right\rfloor=
    \begin{cases}
        \frac{1}{2}+\frac{1}{2d}, &\text{ if $d$ is odd},\\
        \frac{1}{2}+\frac{1}{2(d+1)}, &\text{ if $d$ is even},
    \end{cases}
    \]\
    as we wished to show.
\end{proof}

\subsection{Unit distance graphs}

In this section, we generalise the proof above to show that $\tau(\mathbb{S}^{d-1})\leq G(d)$.

\begin{proof}[Proof of Proposition \ref{prop: bound of tau by G}]
    Let $G$ be a unit distance digraph in $\R^d$ and $A$ be a $\Delta$-free subset of $\mathbb{S}^{d-1}$. Let $\chi$ be the indicator function of $A$ and $\rho\in\operatorname{SO}(d)$ be chosen uniformly at random (with respect to the Haar measure). We then have 
    \begin{equation}\label{eq:mu-of-A}
        \mu_{\mathbb{S}^{d-1}}(A)=\frac{1}{|E(G)|}\E\sum_{xy\in E(G)}\chi(\rho(x-y)).
    \end{equation}
    Similarly to the proof of Theorem~\ref{thm: trangle-free sphere}, for a fixed $\rho\in\operatorname{SO}(d)$, let $G'$ be a subgraph of $G$ such that $xy\in E(G')$ iff $\rho(x-y)\in A$. Since $A$ is $\Delta$-free, $G'$ has no directed triangles. Hence, the sum in Equation~\ref{eq:mu-of-A} is bounded by $\bar{t}(G)$. It follows that $\tau(\mathbb{S}^{d-1})\leq G(d)$.
\end{proof}

Next, we wish to prove Proposition \ref{prop: G lower bound}. Before doing so, we recall some results on the high-dimensional Hadwiger-Nelson problem. The chromatic number $\chi(\mathbb{R}^d)$ is known to grow exponentially, viz.
\[
\left(1+o(1)\right)\cdot1.2^d\leq\chi(\R^d)\leq\left(3+o(1)\right)^d
\]
The lower bound is due to Frankl and Wilson~\cite{Frankl1981} and the upper bound to Larman and Rogers~\cite{Larman_Rogers_1972}. In particular, this means that there is an absolute constant $C>0$ such that $\chi(\mathbb{R}^d)\leq \exp(Cd)$. Consequently, the vertices of any unit distance digraph $G$ in $\mathbb{R}^d$ can be partitioned into $\exp(Cd)$ colour classes, so that there are no edges between vertices of the same colour. We are now ready to prove Proposition \ref{prop: G lower bound}.

\begin{proof}[Proof of Proposition \ref{prop: G lower bound}]
    Let $G$ be a unit distance digraph in $\R^d$. Consider a partition of the vertices of $V(G)$ into $\chi\leq\exp(Cd)$ colour classes so that there are no edges between vertices of the same colour. Let us partition the vertex set $V(G)$ into sets $A$ and $B$ by including a random set of $a=\ceil{\frac{\chi}{2}}$ colour classes in $A$ and the remaining $\chi-a=\floor{\frac{\chi}{2}}$ in $B$. Let $G'$ be the subgraph of $G$ consisting of edges that have one endpoint in $A$ and one in $B$. The probability that a given edge of $G$ is included in $G'$ is 
    \[
        \frac{2\binom{\chi-2}{a-1}}{\binom{\chi}{a}}=\frac{2a(\chi-a)}{\chi(\chi-1)}\geq\frac{\chi^2-1}{2\chi(\chi-1)}\geq\frac{1}{2}(1+\exp(-Cd)).
    \]

    Note that $G'$ is bipartite, and hence has no directed triangles. Further, by the above, $G'$ has $\frac{1}{2}(1+\exp(-Cd))|E(G)|$ edges in expectation.
    Hence, there exists a directed-triangle-free subgraph of $G$ with at least $\frac{1}{2}(1+\exp(-Cd))|E(G)|$ edges, implying $G(d)\geq\frac{1}{2}(1+\exp(-Cd))$ as desired.
\end{proof}

\subsection{Harmonic analysis on the sphere}

Next, we wish to prove Proposition \ref{fourier prop}. The proof uses harmonic analysis on the sphere; we follow the setup of Zakharov~\cite{zakharov25} and refer the reader to ~\cite[Section 2.1]{zakharov25} for background.

\begin{proof}[Proof of Proposition \ref{fourier prop}]
    It suffices to show that a sum-free subset of $\mathbb{S}^{d-1}$ occupies at most $1/2+\exp(-cd)$ of its surface measure for $d\geq d_0$ by integrating over spherical shells. Suppose $A\subset \mathbb{S}^{d-1}$ is sum-free with measure $\mu_{\mathbb{S}^{d-1}}(A)=\alpha$. Let $f=1_A$, and suppose
    \[
    f=\sum_{i=0}^\infty f^{=i},
    \]
    where $f^{=i}$ is the projection of $f$ onto $\mathcal H_{d, i}$, the space of homogeneous harmonic polynomials of degree $i$ in $d$ variables. For $z\in S^{d-1}$, let
    \[
    S(z)\defeq\{y\in \mathbb{S}^{d-1}:\langle y,z\rangle=0.5\},
    \]
    and further, let $\mu_{S(z)}$ denote the uniform $(d-2)$-dimensional surface measure on $S(z)$. Note that for $y\in S(z)$, $z-y\in S(z)$. For $z\in A$, as $A$ is sum-free, $A$ and $z-A$ are disjoint and so $\mu_{S(z)}(A)\leq 0.5$. Let $y, z$ be a uniformly random pair of vectors in $\mathbb{S}^{d-1}$ with inner product $\frac{1}{2}$. Consider the expectation $\mathbb{E}[f(y)f(z)]$. On the one hand, it equals
    \begin{equation}
    \label{eqn: E upper bound}
        \mathbb{E}[f(y)f(z)]=\alpha\mathbb{E}_{z\in A}[\mu_{S(z)}(A)]\leq \frac{\alpha}{2}.
    \end{equation}
    On the other hand,
    \begin{equation}
    \label{eq: E expanded}
    \mathbb{E}[f(y)f(z)]=\sum_{i\geq 0}P_{d,i}(0.5)\|f^{=i}\|_2^2=\alpha^2+\sum_{i\geq 1}P_{d,i}(0.5)\|f^{=i}\|_2^2,
    \end{equation}
    where $P_{d,i}(t)$ is the ultraspherical (or Gegenbauer) polynomial given by
    \[
    P_{d, i}(t) = \frac{1}{{d+i-3 \choose i}}\sum_{\ell = 0}^{[i/2]} (-1)^{\ell} \frac{\Gamma(i - \ell + \frac{d-2}{2})}{\Gamma(\frac{d-2}{2})\ell! (i-2\ell)!} (2t)^{i-2\ell}.
    \]
    It follows from \cite[Theorem 4.1]{Castro-Silva2022-mn} that there exist absolute constants $c>0$ and $d_0$ so that
    \[P_{d,i}(0.5)\geq -\exp(-cd)\]
    holds for all $i>0$ and $d\geq d_0$. Therefore, we have from Equation \eqref{eq: E expanded} that, for $d\geq d_0$,
    \begin{equation}
    \label{eqn: E lower bound}
        \mathbb{E}[f(y)f(z)]\geq \alpha^2-\exp(-cd)\sum_{i\geq 1}\|f^{=i}\|_2^2\geq \alpha^2-\exp(-cd)\alpha,
    \end{equation}
    where the last inequality is a consequence of Parseval's identity $\sum_{i\geq 0}\|f^{=i}\|_2^2=\alpha$. Combining Inequalities \eqref{eqn: E upper bound} and \eqref{eqn: E lower bound} and simplifying, we have the desired result.
\end{proof}

\section{Concluding remarks and open problems}
\label{sec: conc}

\subsection{Improving Theorem \ref{thm: trangle-free sphere} via Conjecture \ref{conj: bound on $G$}}
\label{subsection: conc1}
Recall that Conjecture \ref{conj: bound on $G$} would imply a stronger version of Theorem \ref{thm: trangle-free sphere} if true. We believe that Hamming digraphs are good candidates to witness the conjecture.  Let $V(G)=\{0,1\}^n$, and include $xy\in E(G)$ if $x$ and $y$ differ in exactly $\floor*{\alpha n}$ coordinates for some $\alpha \in (0,0.5)$. After rescaling, this is a unit distance digraph. We believe that for certain ranges of $\alpha$, these digraphs should provide examples witnessing Conjecture \ref{conj: bound on $G$}. Indeed, we believe it is possible to prove the weaker (undirected) version of this conjecture using spectral methods, which would give an alternate proof of Proposition \ref{fourier prop}. However, we do not pursue this here as the harmonic proof we present is much simpler.

\subsection{A stronger version of Conjecture \ref{main conj}}

\begin{question}
    Is it true that for any convex, centrally symmetric region $R\subset\mathbb{R}^d$, $\tau(R)$ equals a half?
\end{question}

The lower bound of half follows by considering a half-space. If the answer to this question is in the affirmative, it would imply Conjecture \ref{main conj} by the following lemma.

\begin{lemma}
    For any convex, centrally symmetric region $R\subset\mathbb{R}^d$, it holds that $\s(R)\leq \ta(R)$.
\end{lemma}
\begin{proof}
    By restricting to the affine span of $R$ if necessary, we can assume that $R$ is full-dimensional. 
    Let $S$ be a sum-free subset of $R$. Let $f$ denote an arbitrary non-zero linear functional on $\mathbb{R}^d$ and let $H$ denote the open half-space $\{x\in \mathbb{R}^d:f(x)>0\}$. As $R$ is full-dimensional, $\mu(R\cap \mathrm{ker}(f))=0$. Replacing $f$ with $-f$ if necessary, we can assume that $S'\defeq S\cap H$ has measure at least $\mu(S')\geq 0.5\mu(S)$.
    
    Note that $H$ is $\Delta$-free as $f(a+b+c)=f(a)+f(b)+f(c)>0$ for $a,b,c\in H$. Further, as $S$ is sum-free, $S'=S\cap H$ must be both $\Delta$-free and sum-free.
    It follows that there are no solutions to $\pm a\pm b \pm c=0$ with $a,b,c\in S'$.
    
    Then $S''\defeq S'\cup -S'\subset R$ must also be $\Delta$-free and $\mu(S'')=2\mu(S')\geq \mu(S)$. Therefore $\ta(R)\geq \s(R)$, as we wished to show.
\end{proof}

\printbibliography
\end{document}